\documentclass[oneside,10pt]{article}          
\usepackage[b5paper]{geometry}	    
\usepackage{amsfonts,amsmath,latexsym,amssymb} 
\usepackage{theorem}                
\usepackage{mathrsfs,upref}         
\usepackage{mathptmx}		    
\usepackage{jmi}	            

\newtheorem{Theorem}{Theorem}           
\newtheorem{Lemma}{Lemma}       
\newtheorem{Proposition}{Proposition}        
\newtheorem{Corollary}{Corollary}

\theoremstyle{definition}

\newtheorem{Definition}{Definition}
\newtheorem{Example}{Example}

\begin{document}

\title[Short title]{The skew generalized von Neumann-Jordan constant in the unit sphere}


\author{Yuxin Wang, Qi Liu$^*$, Jinyu Xia and Shuaizhe Huang}

\address{Yuxin Wang, School of Mathematics and physics, Anqing Normal University, Anqing 246133, P.R.China\\
\email{Y24060028@stu.aqnu.edu.cn}}

\address{Qi Liu, School of Mathematics and physics, Anqing Normal University, Anqing 246133, P.R.China\\
\email{liuq67@aqnu.edu.cn}}

\address{Jinyu Xia,  School of Mathematics and physics, Anqing Normal University, Anqing 246133, P.R.China\\
\email{Y23060036@stu.aqnu.edu.cn}}

\address{Shuaizhe Huang,  School of Mathematics and physics, Anqing Normal University, Anqing 246133, P.R.China\\
	\email{\texttt{drunk\_1011@qq.com}}}

\CorrespondingAuthor{Qi Liu}


\date{DD.MM.YYYY}                               




\begin{abstract}
       	In this paper, we introduce a new constant for Banach spaces, denoted as $\widetilde C_\mathrm{NJ}^{p}(\xi,\nu, X)$. We provide calculations for both the lower and upper bounds of this constant, as well as its exact values in certain Banach spaces. Furthermore, we give the inequality relationship between the $\widetilde C_\mathrm{NJ}^{p}(\xi,\nu, X)$ constant and the other two constants. Besides, we establish an equivalent relationship between the $\widetilde C_\mathrm{NJ}^{p}(\xi,\nu,X)$ constant and the $\widetilde{C}{_\mathrm{N J}^{(p)}(X)}$ constant. Specifically, we shall exhibit the connections between the constant $\widetilde C_\mathrm{NJ}^{p}(\xi,\nu, X)$ and certain geometric characteristics of Banach spaces, including uniform convexity and uniform nonsquareness. Additionally, a sufficient condition for uniform normal structure about the $\widetilde C_\mathrm{NJ}^{p}(\xi,\nu, X)$ constant is also established.
\end{abstract}

\maketitle



\section{Introduction}

      Throughout this article, let $X$ be a real Banach space with dimension at least $2$. We will define $S_X$ as the unit sphere, that is $S_X = \left\{x \in X : \|x\|_X= 1 \right\} $; $ B_X$ as the closed unit ball, that is $B_X = \left\{x \in X : \|x\|_X \leq 1\right\} $ ; and $ ex(B_X)$ as the collection of extreme points of $B_X$.
      
       Geometric constants have been thoroughly developed and are crucial to the theory of Banach spaces. Various geometric constants can help analyze and comprehend the space's qualities and offer useful information. Many geometric constants, such as  $C_\mathrm{N J}(X)$, $C_\mathrm{N J}^{(p)}(X)$  and the $L_\mathrm{Y J}(\xi,\nu, X)$ constant which has a skew connection and so on, have been defined and explored in recent years. Through extensive research on these constants, many scholars have generalized them and explored their connections with other classical constants.
       
       Now let us recall some definitions of geometric constants that will be needed in the following article. 
       
       As the smallest constant $C$, Clarkson\cite{01.1} presented the following von-Neumann constant,   $$\frac{1}{C}\leq\frac{\|x+y\|^2+\|x-y\|^2}{2\|x\|^2+2\|y\|^2}\leq C,$$ where $x,y\in X$ and not both $0$.
     
     Subsequently, the von Neumann-Jordan constant scholars are provided with an equivalent definition by study, as follows:
     
     $$
     C_{\mathrm{NJ}}(X) = \sup \left\{ \frac{\|x + y\|^2 + \|x - y\|^2}{2\|x\|^2 + 2\|y\|^2}:x,y\in X \text{~not both 0}\right\}.
     $$
     \\After extensive research endeavors by scholars, it has been uncovered that $C_{\text{NJ}}(X)$ serves as a valuable tool for characterizing Hilbert spaces, uniformly non-square spaces, and super-reflexive spaces. Additionally, numerous connections have been discovered between $C_{\text{NJ}}(X)$ and certain geometric properties of Banach spaces. For further insights and results related to $C_{\text{NJ}}(X)$, please refer to \cite{09,12}.
     
     In 2008, Alonso et al. \cite{03} give the following constant	$C_\mathrm{NJ}^{\prime}(X)$ based on $C_\mathrm{N J}(X)$, and it's introduced by
     $$
     C_\mathrm{NJ}^{\prime}(X)=\sup \left\{\frac{\|x+y\|_X^2+\|x-y\|_X^2}{4}:x,y\in X, \|x\|_X=\|y\|_X=1 \right\}.
     $$
     Similarly,  Hilbert spaces, uniformly non-square spaces are also characterized by the constant	$C_\mathrm{NJ}^{\prime}(X)$. After this, many scholars have become interested in generalizing the $C_\mathrm{N J}(X)$ constant and the  $C_\mathrm{NJ}^{\prime}(X)$ constant. 
     
     In 2015, Cui et al. \cite{06} defined the following constant denoted as $C{_\mathrm{N J}^{(p)}(X)}$.In fact, it's a generalization of $C_\mathrm{N J}(X)$. $$C_{\mathrm{NJ}}^{(p)}(X)=\sup\left\{\frac{\|x+y\|^{p}+\|x-y\|^{p}}{2^{p-1}\|x\|^{p}+2^{p-1}\|y\|^{p}}:x,y\in X,(x,y)\neq(0,0)\right\},$$ where $1\leq p<+\infty$.
     
     In 2017, Yang, Wang \cite{01} defined a new constant $\widetilde{C}{_\mathrm{N J}^{(p)}(X)}$ motivated by constants $C_\mathrm{NJ}^{\prime}(X)$ and ${C}{_\mathrm{N J}^{(p)}(X)}$. It's introduced by$$\widetilde{C}_{NJ}^{(p)}(X)=\sup\left\{\frac{\|x+y\|_X^p+\|x-y\|_X^p}{2^p}:x,y\in X,\|x\|_X=1,\|y\|_X=1\right\}, $$where  $1\leq p<+\infty$.
     
     Numerous scholars have carried out research on computing the exact values of these two constants in particular Banach spaces(see \cite{20}).
     
     In new research, a new geometric constant with a skew relationship $L_\mathrm{Y J}(\xi,\nu, X)$ for $\xi,\nu>0$ introduced by Liu et al. \cite{04} as $$L_\mathrm{Y J}(\xi,\nu, X)=\sup \left\{\frac{\|\xi x+\nu y\|^2+\|\nu x-\xi y\|^2}{\xi^2(\|x\|^2+\|y\|^2)+\nu^2(\|x\|^2+\|y\|^2)}: x,y\in X, (x,y)\neq (0,0) \right\}.$$
     Another similar constant$$
     L_\mathrm{Y J}^{\prime}(\xi,\nu,X)=\sup \left\{\frac{\|\xi x+\nu y\|^2+\|\nu x-\xi y\|^2}{2(\xi^2+\nu^2)}: x,y \in S_X \right\}.$$
     
     Afterwards, scholars discussed the connections between the $L_\mathrm{Y J}(\xi,\nu, X)$ constant and other constants. Additionally, they calculated the exact values of the $L_\mathrm{Y J}(\xi,\nu, X)$ constant in specific Banach spaces. For further information on the $L_\mathrm{Y J}(\xi,\nu,X)$ constant, please refer to \cite{18,19}.
     
     In recent days, through the generalization of the $L_\mathrm{Y J}(\xi,\nu, X)$ constant. The author \cite{13} introduced a skew generalized von Neumann-Jordan constant denoted as $  C_\mathrm{NJ}^{p}(\xi,\nu,X)$, and it's definition as follows: for $ \xi,\nu>0$,$$C_\mathrm{NJ}^p(\xi,\nu,X)=\sup\left\{\frac{\|\xi x+\nu y\|^p+\|\nu x-\xi y\|^p}{2^{p-2}(\xi^p+\nu^p)(\|x\|^p+\|y\|^p)}:x,y\in X ~\text{not both 0}\right\}, $$where $1\leq p<+\infty$.

     It is evident that $C_\mathrm{NJ}^{\prime}(X)$, $\widetilde{C}{_\mathrm{N J}^{(p)}(X)}$, and $L_\mathrm{Y J}^{\prime}(\xi,\nu, X)$ are all constants that are defined within the unit sphere. Naturally, we will ask whether the skew generalized von Neumann-Jordan constant $  C_\mathrm{NJ}^{p}(\xi,\nu, X)$ can be generalized in the unit sphere. Then the following results are the key to the question. 
     
     The arrangement of this article is as follows: \\In the second part, we will define a new geometric constant, denoted as $\widetilde C_\mathrm{NJ}^{p}(\xi,\nu, X)$. We will provide a calculation for the upper and lower bounds of this constant, as well as its values in some specific Banach spaces. Additionally, we will discuss its relationship with some other constants.\\ In the third part, we will discuss some geometric properties of Banach spaces based on the $\widetilde C_\mathrm{NJ}^{p}(\xi,\nu, X)$ constant.\\In the last part, we shall establish a sufficient condition for uniform normal structure in relation to the $\widetilde C_\mathrm{NJ}^{p}(\xi,\nu, X)$ constant.
     \section{ The constant $\widetilde C_\mathrm{NJ}^{p}(\xi,\nu,X)$ }
     By conducting some studies on the $  C_\mathrm{NJ}^{p}(\xi,\nu, X)$ constant, we obtain the following constant, which is denoted as $\widetilde C_\mathrm{NJ}^{p}(\xi,\nu, X)$.
     \begin{Definition}
     	Let $X$ be a  Banach space, for $\xi,\nu>0$, $$\widetilde C_\mathrm{N J}^p(\xi,\nu,X)=\sup \left\{\frac{\|\xi x+\nu y\|^p+\|\nu x-\xi y\|^p}{2^{p-1}(\xi^p+\nu^p)}: x,y \in S_X\right\},$$where $1\leq p<+\infty$.
     \end{Definition} 
     
     Then let's determine the upper and lower limits of the constant $\widetilde C_\mathrm{NJ}^{p}(\xi,\nu,X)$.
     \begin{Proposition}
     	Let $X$ be a  Banach space, then $$\frac{(\xi+\nu)^{p}+|\nu-\xi|^{p}}{2^{p-1}(\xi^{p}+\nu^{p})}\leq \widetilde C_\mathrm{N J}^p(\xi,\nu,X)\leq\frac{(\xi+\nu)^{p}}{2^{p-2}(\xi^{p}+\nu^{p})}.$$
     	
     \end{Proposition}
     \begin{proof}
     	We prove the right inequality first.\\ According to the convex function $h(x)=\|x\|^p$, we have
     	$$\begin{aligned}
     		&\frac{\|\xi x+\nu y\|^{p}+\|\nu x-\xi y\|^{p}}{(\xi+\nu)^{p}}\\& =\bigg\|\frac{\xi}{\xi+\nu}x+\frac{\nu }{\xi+\nu}y\bigg\|^{p}+\bigg\|\frac{\nu}{\xi+\nu}x+\frac{\xi}{\xi+\nu}(-y)
     		\bigg\|^{p} \\
     		&\leq\frac{\xi}{\xi+\nu}\|x\|^p+\frac{\nu}{\xi+\nu}\|y\|^p+\frac{\nu}{\xi+\nu}\|x\|^p+\frac{\xi}{\xi+\nu}\|y\|^{p}\\&=2.
     	\end{aligned}$$
     	This means that$$\|\xi x+\nu y\|^{p}+\|\nu x-\xi y\|^{p}\leq2(\xi+\nu)^{p},$$
     	hence $$	\frac{\|\xi x+\nu y\|^{p}+\|\nu x-\xi y\|^{p}}{2^{p-1}(\xi^p+\nu^p)}\leq\frac{2(\xi+\nu)^{p}}{2^{p-1}(\xi^{p}+\nu^{p})}=\frac{(\xi+\nu)^{p}}{2^{p-2}(\xi^{p}+\nu^{p})},$$\\which implies that $\widetilde C_\mathrm{N J}^p(\xi,\nu,X)\leq\frac{(\xi+\nu)^{p}}{2^{p-2}(\xi^{p}+\nu^{p})}.$
     	
     	On the other hand, let $y=x$, we have
     	$$\begin{aligned}
     		\frac{\|\xi x+\nu y\|^{p}+\|\nu x-\xi y\|^{p}}{2^{p-1}(\xi^p+\nu^p)}&=\frac{\|\xi x+\nu x\|^{p}+\|\nu x-\xi x\|^{p}}{2^{p-1}(\xi^p+\nu^p)}\\&=\frac{(\xi+\nu)^{p}+|\nu-\xi|^{p}}{2^{p-1}(\xi^{p}+\nu^{p})}.
     	\end{aligned}$$
     	
     	Hence $$\widetilde C_\mathrm{N J}^p(\xi,\nu,X)\geq\frac{(\xi+\nu)^{p}+|\nu-\xi|^{p}}{2^{p-1}(\xi^{p}+\nu^{p})}.$$
     \end{proof}

     \begin{Example}
     	(  $l_1$ spaces)	Let $X=(R^2,\|\cdot\|_1)$, then $\widetilde C_\mathrm{N J}^p(\xi,\nu,X)=\frac{(\xi+\nu)^{p}}{2^{p-2}(\xi^{p}+\nu^{p})}.$
     \end{Example}
     \begin{proof}
     	Let $x=(0,1)$, $y=(0,-1)$, we obtain$$\begin{aligned}\widetilde C_\mathrm{N J}^p(\xi,\nu,l_1)&\geq\frac{\|\xi x+\nu y\|_{1}^{p}+\|\nu x-\xi y\|_{1}^{p}}{2^{p-1}(\xi^{p}+\nu^{p})}\\&=\frac{2(\xi+\nu)^{p}}{2^{p-1}(\xi^{p}+\nu^{p})}\\&=\frac{(\xi+\nu)^{p}}{2^{p-2}(\xi^{p}+\nu^{p})}.\end{aligned}$$
     	Therefore by the upper bound of the constant, $\widetilde C_\mathrm{N J}^p(\xi,\nu,l_1)=\frac{(\xi+\nu)^{p}}{2^{p-2}(\xi^{p}+\nu^{p})}$ holds .
     \end{proof}
     \begin{Example}
     	(  $l_{\infty}$ spaces)	Let $X=(R^2,\|\cdot\|_\infty)$,  then  $\widetilde C_\mathrm{N J}^p(\xi,\nu,X)=\frac{(\xi+\nu)^{p}}{2^{p-2}(\xi^{p}+\nu^{p})}.$
     \end{Example}
     \begin{proof}
     	Let $x=(1,1)$, $y=(1,-1)$, we obtain$$\begin{aligned}\widetilde C_\mathrm{N J}^p(\xi,\nu,l_{\infty})&\geq\frac{\|\xi x+\nu y\|_{\infty}^{p}+\|\nu x-\xi y\|_{\infty}^{p}}{2^{p-1}(\xi^{p}+\nu^{p})}\\&=\frac{2(\xi+\nu)^{p}}{2^{p-1}(\xi^{p}+\nu^{p})}\\&=\frac{(\xi+\nu)^{p}}{2^{p-2}(\xi^{p}+\nu^{p})}.\end{aligned}$$
     	Therefore by the upper bound of the constant, $\widetilde C_\mathrm{N J}^p(\xi,\nu,l_{\infty})=\frac{(\xi+\nu)^{p}}{2^{p-2}(\xi^{p}+\nu^{p})}$ holds .
     \end{proof}
     
     \begin{Example}
     	
     	(  $l_1-l_{\infty}$ spaces). Let $X=l_1-l_{\infty}$ which is  $\mathbb{R}^2$  endowed with the norm 
     	$$\|x\|=\left\{\begin{array}{l}
     		\|x\|_{\infty}, \text { if } x_1 x_2 \geq 0, \\
     		\|x\|_1, \text { if } x_1 x_2 \leq 0.
     	\end{array}\right.$$
     	Then if $\xi\geq\nu$, $$\widetilde C_\mathrm{NJ}^{p}(\xi,\nu,l_{\infty}-l_{1})=\frac{\xi^p+(\xi^p+\nu^p)}{2^{p-1}(\xi^p+\nu^p)}.$$
     \end{Example}
     
     \begin{proof}
     	For any $x,y$ that belong to $S_X$, this can be expressed in the form of  extreme points as follows:
     	$$x=\lambda x_1+(1-\lambda)x_2\text{ and }y=\mu y_1+(1-\mu)y_2,$$
     	where $\lambda,\mu\in[0,1]$ and $x_1,x_2,y_1$, $y_2$ are extreme points of the corresponding line segment.
     	
     	According to Minkowski inequality, we have$$\begin{aligned}
     		&\|\xi x+\nu y\|^{p}+\|\nu x-\xi y\|^{p}\\&=\|\lambda(\xi x_1+\nu y)+(1-\lambda)(\xi x_2+\nu y)\|^p+\|\lambda(\nu x_1-\xi y)+(1-\lambda)(\nu x_2-\xi y)\|^p
     		\\&\leq\lambda\|\xi x_1+\nu y\|^p+(1-\lambda)\|\xi x_2+\nu y\|^p+\lambda\|\nu x_1-\xi y\|^p+(1-\lambda)\|\nu x_2-\xi y\|^p\\&=\lambda\|\mu(\xi x_1+\nu y_1)+(1-\mu)(\xi x_1+\nu y_2)\|^p+\lambda\|\mu(\nu x_1-\xi y_1)+(1-\mu)(\nu x_1-\xi y_2)\|^p\\&+
     		(1-\lambda)\|\mu(\xi x_{2}+\nu y_{1})+(1-\mu)(\xi x_{2}+\nu y_{2})\|^{p} \\
     		&+(1-\lambda)\|\mu\left(\nu x_{2}-\xi y_{1}\right)+(1-\mu)\left(\nu x_{2}-\xi y_{2}\right)\|^{p} \\
     		&\leqslant\lambda\mu\left[\|\xi x_{1}+\nu y_{1}\|^{p}+\|\nu x_{1}-\xi y_{1}\|^{p}\right]+\lambda(1-\mu)\left[\|\xi x_{1}+\nu y_{2}\|^{p}+\|\nu x_{1}-\xi y_{2}\|^{p}\right] \\
     		&+(1-\lambda)\mu\left[\|\xi x_2+\nu y_1\|^p+\|\nu x_2-\xi y_1\|^p\right] \\
     		&+(1-\lambda)(1-\mu)\left[\|\xi x_{2}+\nu y_{2}\|^{p}+\|\nu x_{2}-\xi y_{2}\|^{p}\right],
     	\end{aligned}$$
     	which means that $$\begin{aligned}
     		\|\xi x+\nu y\|^{p}+\|\nu x-\xi y\|^{p}\leqslant\operatorname*{max}\{\|\xi x_{1}+\nu y_{1}\|^{p}+\|\nu x_{1}-\xi y_{1}\|^{p}, \\
     		\|\xi x_{1}+\nu y_{2}\|^{p}+\|\nu x_{1}-\xi y_{2}\|^{p} , \\
     		\|\xi x_{2}+\nu y_{1}\|^{p}+\|\nu x_{2}-\xi y_{1}\|^{p} , \\
     		\|\xi x_{2}+\nu y_{2}\|^{p}+\|\nu x_{2}-\xi y_{2}\|^{p}\}.
     	\end{aligned}$$
     	Thus, we only need to take into account the exact values of the constant at extreme points. Since $ex(B_X)=\{(1,0),(0,1),(-1,0),(0,-1),(1,1),(-1,-1)\}$ by calculation and the above $x$ can be replaced by $-x$, the above $y$ can be replaced by $-y$ by observation, it is only necessary to consider the case when $x,y=(1,0),(0,1)$ or $(1,1)$. \\According to the inequality $\xi\geq\nu$, we can substitute the values of $x$ and $y$ given above into the calculation to find that the expression $\|\xi x+\nu y\|^{p}+\|\nu x-\xi y\|^{p}$ reaches its maximum at one of the following points: $$\begin{cases}x'_1=(0,1),\\y'_1=(1,0),\end{cases} \begin{cases}x'_2=(0,1),\\y'_2=(1,1),\end{cases} \text{or} \begin{cases}x'_3=(1,0),\\y'_3=(1,1),\end{cases}$$ additionally, it is always true that $\|\xi x+\nu y\|^{p}+\|\nu x-\xi y\|^{p}\leq\xi^p+(\xi^p+\nu^p)$.
     	
     	Therefore $\widetilde C_\mathrm{NJ}^{p}(\xi,\nu,l_{1}-l_{\infty})=\frac{\xi^p+(\xi^p+\nu^p)}{2^{p-1}(\xi^p+\nu^p)}$ holds.
     \end{proof}
     
     Once the upper and lower bounds of the $\widetilde C_\mathrm{NJ}^{p}(\xi,\nu, X)$ constant have been calculated and examples of the constant in particular Banach spaces have been listed, we investigate the relationships between this constant and other constants in more detail. The skew generalized von Neumann-Jordan constant $ C_\mathrm{NJ}^{p}(\xi,\nu,X)$ and the constant $\widetilde C_\mathrm{NJ}^{p}(\xi,\nu,X)$ will be connected, as will the $\widetilde{C}{_\mathrm{N J}^{(p)}(X)}$ constant. This certainly gives us a fresh viewpoint for a more thorough comprehension and application of geometric constants.

     Recall an inequality between von Neumann-Jordan constant $C_\mathrm{N J} (X)$ and the modified $\mathrm{N J}$ constant $C'_\mathrm{N J} (X)$ was proved by Alonso et al. \cite{03} as $$C_{\mathrm{NJ}}(X)\leq2(1+C_{\mathrm{NJ}}'(X)-\sqrt{2C_{\mathrm{NJ}}'(X)})\leq2.$$
     So it's necessary for us to discuss the relationship between the skew generalized von Neumann-Jordan constant $  C_\mathrm{NJ}^{p}(\xi,\nu, X)$ and the constant $\widetilde C_\mathrm{NJ}^{p}(\xi,\nu, X)$. Then the next theorem asks the question.
     
     \begin{Theorem}
     	Let $X$ be a Banach space, for $p>1$ and $\frac{1}{p}+\frac{1}{q}=1$. Then$$\widetilde{C}_{\mathrm{NJ}}^{p}(\xi,\nu,X)\leq C_{\mathrm{NJ}}^{p}(\xi,\nu,X)\leq2^{2-p}[1+(2^{\frac{1}{q}}(\widetilde{C}_{\mathrm{NJ}}^{p}(\xi,\nu,X))^{1/p}-1)^{q}]^{p-1}.$$
     	
     \end{Theorem}
     \begin{proof}
     	On the one hand, we have$$\frac{\|\xi x+\nu y\|^{p}+\|\nu x-\xi y\|^{p}}{2^{p-2}(\xi^{p}+\nu ^{p})(\|x\|^{p}+\|y\|^{p})}\geq\frac{\|\xi x+\nu y\|^{p}+\|\nu x-\xi y\|^{p}}{2^{p-1}(\xi^{p}+\nu^{p})},$$ which means that $\widetilde{C}_{\mathrm{NJ}}^{p}(\xi,\nu,X)\leq C_{\mathrm{NJ}}^{p}(\xi,\nu,X)$.\\However, without losing generality, we can assume that $0 \leq \|y\| \leq \|x\| = 1$ in the definition of $ C_\mathrm{NJ}^{p}(\xi,\nu,X)$. After that, we examine the scenario in which $0< \|y\| \leq \|x\| = 1$.
     	\\Since$$\left\|\xi x+\nu\frac{y}{\|y\|}\right\|=\left\|\frac{\xi x}{\|y\|}+\frac{\nu y}{\|y\|}+\xi x-\frac{\xi x}{\|y\|}\right\|\geq\left|\frac{\|\xi x+\nu y\|+\xi(\|y\|-1}{\|y\|}\right|,$$and$$\left\|\nu x-\xi\frac{y}{\|y\|}\right\|=\left\|\frac{\nu x}{\|y\|}-\frac{\xi y}{\|y\|}+\nu x-\frac{\nu x}{\|y\|}\right\|\geq\left|\frac{\|\nu x+\xi y\|+\nu(\|y\|-1)}{\|y\|}\right|.$$We obtain$$\begin{aligned}
     		\widetilde C_{\mathrm{NJ}}^p(\xi,\nu,X)& \geq\frac{\left\|\xi x+\nu\frac{y}{\|y\|}\right\|^{p}+\left\|\nu x-\xi \frac{y}{\|y\|}\right\|^{p}}{2^{p-1}(\xi^p+\nu^p) }\\
     		&\geq\frac{|\|\xi x+\nu y\|+\xi(\|y\|-1)|^{p}+|\|\nu x-\xi y\|+\nu(\|y\|-1)|^{p}}{2^{p-1}\|y\|^{p}(\xi^p+\nu^p)},
     	\end{aligned}$$and from the inequality$$
     	\begin{aligned}
     		&[|\|\xi x+\nu y\|+\xi(\|y\|-1)|^p+|\|\nu x-\xi y\|+\nu(\|y\|-1)|^p]^{\frac{1}{p}}\\&+(\xi^p(1-\|y\|)^p+\nu^p(1-\|y\|)^p)^{\frac{1}{p}}\\&\geq(\|\xi x+\nu y\|^p+\|\nu x-\xi y\|^p)^{\frac{1}{p}},
     	\end{aligned}$$we have$$\begin{aligned}
     		&[2^{p-1}\|y\|^{p}(\xi^{p}+n^{p})\widetilde C_{\mathrm{NJ}}^p(\xi,\nu,X)]^{\frac{1}{p}}\\& \geq[|\|\xi x+\nu y\|+\xi(\|y\|-1)|^{p}+|\|\nu x-\xi y\|+\nu(\|y\|-1)|^{p}]^{\frac{1}{p}} \\
     		&\geq(\|\xi x+\nu y\|^{p}+\|\nu x-\xi y\|^{p})^{\frac{1}{p}}-[\xi^{p}(1-\|y\|)^{p}+\nu^{p}(1-\|y\|)^{p}]^{\frac{1}{p}},
     	\end{aligned}$$ hence $$\begin{aligned}&2^{1-\frac{1}{p}}\|y\|(\widetilde C_{\mathrm{NJ}}^p(\xi,\nu,X))^{\frac{1}{p}}(\xi^{p}+\nu^{p})^{\frac{1}{p}}+(\xi^{p}+\nu^{p})^{\frac{1}{p}}(1-\|y\|)\\&\geq(\|\xi x+\nu y\|^{p}+\|\nu x-\xi y\|^{p})^{\frac{1}{p}}\\&=2^{1-\frac{2}{p}}(\xi^{p}+\nu^{p})^{\frac{1}{p}}(1+\|y\|^{p})^{\frac{1}{p}}(\frac{\|\xi x+\nu y\|^{p}+\|\nu x-\xi y\|^{p}}{2^{p-2}(\xi ^{p}+\nu^{p})(1+\|y\|)^{p}})^{\frac{1}{p}}\\&=2^{1-\frac{2}{p}}(\xi^{p}+\nu^{p})^{\frac{1}{p}}[1+\|y\|^{p}]^{\frac{1}{p}}\widetilde C_{\mathrm{NJ}}^p(\xi,\nu,X),\end{aligned}$$ we obtain\begin{equation}\label{e1}
     		\begin{aligned}(C_{\mathrm{NJ}}^p(\xi,\nu,X))^{\frac{1}{p}}&\leq\frac{2^{1-\frac{1}{p}}\|y\|(\widetilde C_{\mathrm{NJ}}^p(\xi,\nu,X))^{\frac{1}{p}}+(1-\|y\|)}{2^{1-\frac{2}{p}}[1+\|y\|^{p}]^{\frac{1}{p}}}\\&=\frac{[2(\widetilde C_{\mathrm{NJ}}^p(\xi,\nu,X)]^{\frac{1}{p}}\|y\|+2^{\frac{2}{p}-1}(1-\|y\|)}{(1+\|y\|^{p})^{\frac{1}{p}}}.\end{aligned}
     	\end{equation} 
     	Since $\frac{1}{2^{p-1}}\leq C_\mathrm{N J}^p(\xi,\nu,X)\leq 2$ has been verified in \cite{13}, then we know that the inequality (\ref*{e1}) is also true for $y=0$.\\Now we define a function $f(t)$  as $$f(t)=\frac{[2(\widetilde C_{\mathrm{NJ}}^p(\xi,\nu,X)]^{\frac{1}{p}}t+2^{\frac{2}{p}-1}(1-t)}{(1+t^{p})^{\frac{1}{p}}},$$ where  $0\leq t\leq 1$.
     	
     	Then $f(t)$ attains it's maximum at $t'=[2(\widetilde C_{\mathrm{NJ}}^p(\xi,\nu,X)]^{\frac{1}{p}}-2^{\frac{2}{p}-1}$, thus we have $$(C_{\mathrm{NJ}}^p(\xi,\nu,X))^{\frac{1}{p}}\leq f(t')=2^{\frac{2}{p}-1}[1+(2^{\frac{1}{q}}(\widetilde{C}_{\mathrm{NJ}}^{p}(\xi,\nu,X))^{1/p}-1)^{q}]^{1-\frac{1}{p}}.$$ This means that  $$C_{\mathrm{NJ}}^{p}(\xi,\nu,X)\leq2^{2-p}[1+(2^{\frac{1}{q}}(\widetilde{C}_{\mathrm{NJ}}^{p}(\xi,\nu,X))^{1/p}-1)^{q}]^{p-1}$$.\\Therefore the proof of this theorem is completed.
     \end{proof}
     We derive the following results by taking into account the link between the constants $\widetilde{C}{_\mathrm{N J}^{(p)}(X)}$ and $\widetilde C_\mathrm{NJ}^{p}(\xi,\nu, X)$. Theorem \ref{4.2} provides an inequality between the two constants, while Theorem \ref{4.3} states that there is some equivalence between these two constants.
     \begin{Lemma}\label{lemma}
     	There are two functions: $f_1(t)=\|x+t y\|^p+\|t x-y\|^p$ and $f_2(t)=\|tx+ y\|^p+\| x-ty\|^p$,  they are both convex functions of $t$ on $\mathbb{R}$.
     \end{Lemma}
     \begin{proof}
     	Let $t_1, t_2 \in \mathbb{R}$, $\mu \in(0,1)$ , since $h(x)=\|x\|^p$ is a convex function, we have$$
     	\begin{aligned}
     		&f_1(\mu x_1+(1-\mu)x_2)\\&=\left\|x+\left(\mu t_1+(1-\mu) t_2\right) y\right\|^p+\left\|\left(\mu t_1+(1-\mu) t_2\right) x-y\right\|^p\\&=\|\mu(x+t_1y)+(1-\mu)(x+t_2y)\|^p+\|\mu(t_1x-y)+(1-\mu)(t_2x-y)\|^p
     		\\
     		&\leq  \mu\left\|x+t_1 y\right\|^p+(1-\mu)\left\|x+t_2 y\right\|^p+\mu\left\|t_1 x-y\right\|^p+(1-\mu)\left\|t_2 x-y\right\|^p \\
     		&=  \mu\left(\left\|x+t_1 y\right\|^p+\left\|t_1 x-y\right\|^p\right)+(1-\mu)\left(\left\|x+t_2 y\right\|^p+\left\|t_2 x-y\right\|^p\right) \\
     		&=  \mu f_1\left(t_1\right)+(1-\mu) f_1\left(t_2\right),
     	\end{aligned}
     	$$
     	which means that $f_1(t)$ is a convex function.\\Similarly, we can also prove that	$f_2(\mu x_1+(1-\mu)x_2)\leq
     	\mu f_2\left(t_1\right)+(1-\mu) f_2\left(t_2\right),$	this means that $f_2(x)	$ is a convex function.
     \end{proof}
     
     \begin{Theorem}\label{4.2}
     	Let $X$ be a Banach space. Then$$\widetilde{C}_\mathrm{NJ}^p(\xi, \nu, X) \geqslant \frac{2[{\min}\{\xi, \nu\}]^p}{\xi^p+\nu^p} \widetilde{C}_\mathrm{NJ}^{(p)}(X) .$$
     \end{Theorem}
     \begin{proof}
     	
     	Consider the functions $f_1: [0,+\infty] \rightarrow [0,+\infty]$  defined as
     	$$f_1(t)=\|x+t y\|^p+\|t x-y\|^p,$$ and $f_2: [0,+\infty] \rightarrow [0,+\infty]$  defined as
     	$$f_2(t)=\|tx+y\|^p+\|x-ty\|^p.$$
     	Notice that $f_1(1)=f_2(1)=\|x+y\|^p+\|x-y\|^p$.\\By Lemma \ref{lemma}, we know that $f_1(t),f_2(t)$ are both convex functions. Then we consider the following two cases.\\\textbf{Case 1}: If $\xi\leq\nu$. Since $f_1(t)$ is a convex function, then $f_1(\frac{\nu}{\xi})\geq f_1(1)$ is always valid for $\xi\leq\nu$.\\This means that for any $x,y\in S_X$, we have \begin{equation}\label{e2}
     		\begin{aligned}
     			\frac{\xi^p\left[\left\|x+\frac{\nu}{\xi} y\right\|^p+\left\|\frac{\nu}{\xi} x-y\right\|^p\right]}{2^{p-1}\left(\xi^p+\nu^p\right)}&\geqslant \frac{\xi^p \left(\|x+y\|^p+\|x-y\|^p\right)}{2^{p-1}\left(\xi^p+\eta^p\right)}\\&= \frac{2 \xi^p}{\xi^p+\nu^p} \cdot \frac{\|x+y\|^p+\|x-y\|^p}{2^p}.
     		\end{aligned}
     	\end{equation}
     	
     	Since there is a equivalent definition of the constant $\widetilde C_\mathrm{NJ}^{p}(\xi,\nu,X)$ which described as$$\widetilde C_\mathrm{N J}^p(\xi,\nu,X)=\sup \left\{\frac{\xi^p(\|\ x+\frac{\nu}{\xi} y\|^p+\|\frac{\nu}{\xi}x-
     		y\|^p)}{2^{p-1}(\xi^p+\nu^p)}: x,y \in S_X\right\},$$where $1\leq p<+\infty$, then$$\widetilde{C}_\mathrm{NJ}^p(\xi,\nu, X) \geqslant \frac{2\xi^p}{\xi^p+\nu^p} \widetilde{C}_\mathrm{NJ}^{(p)}(X) ,$$ holds by the inequality (\ref*{e2}).\\\textbf{Case 2}: If $\xi\geq\nu$. Since $f_2(t)$ is a convex function, then $f_2(\frac{\xi}{\nu})\geq f_2(1)$ is always valid for $\xi\geq\nu$.\\This means that for any $x,y\in S_X$, we have \begin{equation}\label{e3}\begin{aligned}
     			\frac{\nu^p\left[\left\|\frac{\xi}{\nu}x+ y\right\|^p+\left\| x-\frac{\xi}{\nu}y\right\|^p\right]}{2^{p-1}\left(\xi^p+\nu^p\right)}&\geqslant \frac{\nu^p \left(\|x+y\|^p+\|x-y\|^p\right)}{2^{p-1}\left(\xi^p+\nu^p\right)}\\&= \frac{2 \nu^p}{\xi^p+\nu^p} \cdot \frac{\|x+y\|^p+\|x-y\|^p}{2^p}.
     		\end{aligned}
     	\end{equation}\\Since there is an another equivalent definition of the constant $\widetilde C_\mathrm{NJ}^{p}(\xi,\nu,X)$ which described as$$\widetilde C_\mathrm{N J}^p(\xi,\nu, X)=\sup \left\{\frac{\nu^p(\|\ \frac{\xi}{\nu}x+ y\|^p+\|x-
     		\frac{\xi}{\nu} y\|^p)}{2^{p-1}(\xi^p+\nu^p)}: x,y \in S_X\right\},$$where $1\leq p<+\infty$, then$$\widetilde{C}_\mathrm{NJ}^p(\xi, \nu, X) \geqslant \frac{2\nu^p}{\xi^p+\nu^p} \widetilde{C}_\mathrm{NJ}^{(p)}(X) ,$$ holds by the inequality (\ref*{e3}).\\Combining the above two cases, we obtain$$\widetilde{C}_\mathrm{NJ}^p(\xi, \nu, X) \geqslant \frac{2[{\min}\{\xi, \nu\}]^p}{\xi^p+\nu^p} \widetilde{C}_\mathrm{NJ}^{(p)}(X) .$$
     	
     \end{proof}
     
     \begin{Theorem}\label{4.3}
     	Let $X$ be a Banach space, these three conclusions are on an equal foot :
     	
     	\qquad $(i)~~ \widetilde{C}{_\mathrm{N J}^{(p)}(X)}=2.$
     	
     	\qquad $(ii)~~ \widetilde C_\mathrm{NJ}^{p}(\xi,\nu, X)=\frac{(\xi+\nu)^{p}}{2^{p-2}(\xi^{p}+\nu^{p})}~\mathrm{for~all}~\xi, \nu.$
     	
     	\qquad $(iii)~~ \widetilde C_\mathrm{NJ}^{p}(\xi_0,\nu_0, X)=\frac{(\xi+\nu)^{p}}{2^{p-2}(\xi^{p}+\nu^{p})} ~\mathrm{for~some}~ \xi_{0}, \nu_{0}.$
     \end{Theorem}
     \begin{proof}
     	Obviously, the foregoing results are valid for $\xi=\nu$, therefore we consider $\xi<\nu$ and $\nu<\xi$. If we assume $\nu<\xi$ without losing generality, then the case $\xi<\nu$ is comparable.\\$(i)\Rightarrow(ii)$. Since $\widetilde{C}{_\mathrm{N J}^{(p)}(X)}=2$, we can deduce that there exists sequence of points $\left\{x_n\right\}$,~$\left\{y_n\right\}\in{S_X},$ such that$$\|x_n+y_n\|\to2, ~\|x_n-y_n\|\to2 ~(n\to\infty).$$\\So we have$$\begin{aligned}
     		||\xi x_{n}+\nu y_{n}||^{p}& =\|\xi (x_{n}+y_{n})-(\xi -\nu)y_{n})\|^{p} \\
     		&\geq[\xi||x_{n}+y_{n}||-|\xi-\nu|\|y_{n}\|]^{p}\\&=[2\xi-(\xi-\nu)]^{p}=(\xi+\nu)^{p},
     	\end{aligned}$$ and 
     	$$\begin{aligned}
     		\|\nu x_n-\xi y_n\|^p&=\|(\xi+\nu)(x_{n}-y_{n})+\nu y_{n}-\xi x_{n}\|^{p} \\&\geq[(\xi+\nu)\|x_{n}-y_{n}\|-\nu\|y_{n}\|-\xi\|x_{n}\|]^{p} \\&=[2(\xi+\nu)-\nu-\xi]^{p} =(\xi+\nu)^p.
     	\end{aligned}$$\\This means that there exists sequence of points $\left\{x_n\right\}$,~$\left\{y_n\right\}\in{S_x}$, such that$$\|\xi x_n+\nu y_n\|^p\to(\xi+\nu)^p, ~\|\nu x_n-\xi y_n\|^p\to(\xi+\nu)^p ~(n\to\infty).$$\\ This implies that $\widetilde C_\mathrm{NJ}^{p}(\xi,\nu,X)=\frac{(\xi+\nu)^{p}}{2^{p-2}(\xi^{p}+\nu^{p})}.$\\$(ii)\Rightarrow(iii)$. Obviously.\\$(iii)\Rightarrow(i)$.We use proof by contradiction to prove this conclusion. \\It's known that $\widetilde{C}{_\mathrm{N J}^{(p)}(X)}\leq 2$, if $\widetilde{C}{_\mathrm{N J}^{(p)}(X)}\neq2$, then we have $\widetilde{C}{_\mathrm{N J}^{(p)}(X)}<2$, which means that there exists $\delta>0$, such that either $\|\frac{x+y}2\|\leq1-\delta$,  or $\|\frac{x-y}2\|\leq1-\delta$ for any $x,y\in S_X$ . Then we can assume $\|\frac{x-y}2\|\leq1-\delta$, $\|\frac{x+y}2\|=1$ without loss of generality(the other case is similar). Then, for some $\xi_0,\nu_0$, we have$$\begin{aligned}
     		\|\xi_{0}x+\nu_{0}y\|^{p}+\|\nu_{0}x-\xi_{0}y\|^{p}&\leq[\|\xi_{0}x\|+\|\nu_{0}y\|]^{p}+\|\nu_{0}(x-y)+(\nu_{0}-\xi_{0})y\|^{p} \\&\leq(\xi_{0}+\nu_{0})^{p}+[2(1-\delta)\nu_{0}+\xi_{0}-\nu_{0}]^{p} \\&=(\xi_{0}+\nu_{0})^{p}+(\xi_{0}+\nu_{0}-2\delta \nu_{0})^{p} \\&<2(\xi_0+\nu_{0})^{p},
     	\end{aligned}$$
     	which means that $\widetilde C_\mathrm{NJ}^{p}(\xi,\nu,X)<\frac{(\xi+\nu)^{p}}{2^{p-2}(\xi^{p}+\nu^{p})}.$This is a contradiction. Then the proof is completed.

     \end{proof}
     \section{ The relations with the constant $\widetilde C_\mathrm{NJ}^{p}(\xi,\nu,X)$ and some geometric properties of Banach spaces}
     Before we give the geometric properties of Banach spaces exhibited by the constant $\widetilde C_\mathrm{NJ}^{p}(\xi,\nu, X)$, we introduce some relations with the constant $\widetilde C_\mathrm{NJ}^{p}(\xi,\nu, X)$ and $\delta_X{(\varepsilon)}$ and the connection between the constant $\widetilde C_\mathrm{NJ}^{p}(\xi,\nu, X)$ and $J(X)$. These linkages serve as a solid foundation for exploring the connection between the constant $\widetilde C_\mathrm{NJ}^{p}(\xi,\nu, X)$ and various geometric characteristics of Banach spaces.

     Now let us recall the definitions of the modulus of convexity  $\delta_X{(\varepsilon)}$  and the James constant $J(X)$ first.
     \begin{Definition}
     	Let $X$ be a Banach space, the definition of James constant $J(X) $\cite{11}  as shown below: $$J(X)=\sup\left\{\min\{\|x+y\|_X,\|x-y\|_X\}:x,y\in X,\|x\|_X=1,\|y\|_X=1\right\}.$$
     \end{Definition}
     \begin{Definition}
     	Clarkson introduced the convexity modulus $\delta_X{(\epsilon)}$ in \cite{16}  as follows: $$\delta_X(\epsilon)=\inf\left\{1-\frac{\|x+y\|}{2}:x,y\in S_X,\|x-y\|\geq\epsilon\right\}, (0\leq\epsilon\leq2),$$
     	and the convexity feature of a Banach space $X$ is described as the number
     	$$\epsilon_{0}(X)=\sup\big\{\epsilon\in[0,2]: \delta_{X}(\epsilon)=0\big\}.$$
     \end{Definition}
     Subsequently, we acquire the following assertion regarding the connection between the constant $\widetilde C_\mathrm{NJ}^{p}(\xi,\nu, X)$ and $\delta_X{(\epsilon)}$.
     \begin{Theorem}\label{section:Theorem}
     	Let $X$ be a Banach space. Then for any $0\leq\epsilon\leq2$.
     	$$\begin{aligned}
     		&\widetilde C_\mathrm{NJ}^{p}(\xi,\nu,X)\\&=\sup\left\{\frac{[2\min\{\xi,\nu\}(1-\delta_X{(\epsilon)}+|\xi-\nu|]^{p}+[\min\{\xi,\nu\}\epsilon+|\xi-\nu|]^p}{2^{p-1}(\xi^{p}+\nu^{p})}\right\}.
     	\end{aligned}
     	$$
     \end{Theorem}
     \begin{proof}
     	According to $$\begin{cases}\|\xi x+\nu y\|\leq \xi\|x+y\|+|\xi-\nu |\\\|\xi x+\nu y\|\leq \eta\|x+y\|+|\xi-\nu|\end{cases},$$and$$\begin{cases}\|\nu x-\xi y\|\leq \nu\|x-y\|+|\xi-\nu |\\\|\nu x-\xi y\|\leq \xi\|x-y\|+|\xi-\nu|\end{cases},$$which means that
     	$$\|\xi x+\nu y\|\leq\min\{\xi,\nu\}\|x+y\|+|\xi-\nu|,$$and$$\|\nu x-\xi y\|\leq\min\{\xi,\nu\}\|x-y\|+|\xi-\nu|.$$Obviously, noticng that for any $x,y\in S_X$, $\delta_X(\|x-y\|)<1-\frac{\|x+y\|}{2}$ is always holds . \\\\Hence we obtain for any $0\leq\epsilon\leq2$,
     	$$\begin{aligned}&\frac{\|\xi x+\nu y\|^p+\|\nu x-\xi y\|^p}{2^{p-1}(\xi^p+\nu^p)}\\&\leq\frac{[\min\left\{\xi,\nu\right\}\|x+y\|+|\xi-\nu|]^{p}+[\min\{\xi,\nu\}\|x-y\|+|\xi-\nu|]^p}{2^{p-1}(\xi^{p}+\nu^{p})}\\&\leq\sup\left\{\frac{[2\min\{\xi,\nu\}(1-\delta_X{(\epsilon)}+|\xi-\nu|]^{p}+[\min\{\xi,\nu\}\epsilon+|\xi-\nu|]^p}{2^{p-1}(\xi^{p}+\nu^{p})}\right\}	.\end{aligned}$$
     	Therefore for any $0\leq\epsilon\leq2$	\begin{equation}\label{e4}\begin{aligned}
     			&\widetilde C_\mathrm{NJ}^{p}(\xi,\nu,X)\\&\leq\sup\left\{\frac{[2\min\{\xi,\nu\}(1-\delta_X{(\epsilon)}+|\xi-\nu|]^{p}+[\min\{\xi,\nu\}\epsilon+|\xi-\nu|]^p}{2^{p-1}(\xi^{p}+\nu^{p})}\right\}.
     		\end{aligned}
     	\end{equation}\\On the contrary, suppose $\epsilon\in[0,2]$. Then, for any $\lambda>0$, there can be found $x,y\in S_X$, such that
     	$$
     		\|x-y\|\geq\epsilon,\quad1-\frac{\|x+y\|}{2}\leq\delta_X(\epsilon)+\lambda.$$Then we have$$\widetilde C_\mathrm{NJ}^{p}(\xi,\nu,X)\geq\frac{[2\min\{\xi,\nu\}(1-\delta_X{(\epsilon)-\lambda)}+|\xi-\nu|]^{p}+[\min\{\xi,\nu\}\epsilon+|\xi-\nu|]^p}{2^{p-1}(\xi^{p}+\nu^{p})}.
$$
     	
     	Since $\lambda$ can be arbitrarily small, let $\lambda\to0$,we obtain
     	\begin{equation}\label{e5}
     		\widetilde C_\mathrm{NJ}^{p}(\xi,\nu,X)\geq\frac{[2\min\{\xi,\nu\}(1-\delta_X{(\epsilon))}+|\xi-\nu|]^{p}+[\min\{\xi,\nu\}\epsilon+|\xi-\nu|]^p}{2^{p-1}(\xi^{p}+\nu^{p})}.
     	\end{equation}
     	Combining the above inequality (\ref*{e4}) and (\ref*{e5}), we have completed the entire proof of Theorem \ref{section:Theorem}.
     \end{proof}
     
     \begin{Example}
     	Let $p>0,r\geq2.$ Then
     	
     	$$\widetilde C_{\mathrm{NJ}}^{(p)}(l_r)=\widetilde C_{\mathrm{NJ}}^{(p)}(L_r)=\begin{cases}\frac{[2^{1-\frac{1}{r}}\min\left\{\xi,\nu\right\}+|\xi-\nu|]^{p}}{2^{p-2}(\xi^{p}+\nu^{p})}&\quad p<r;\\\frac{(\xi+\nu)^{p}}{2^{p-1}(\xi^{p}+\nu^{p})}&\quad p\ge r.\end{cases}$$
     \end{Example}
     \begin{proof}
     	Since $\delta_{l_{r}}(\varepsilon)=\delta_{L_{r}}(\varepsilon)=1-\left[1-\left(\frac{\varepsilon}{2}\right)^{r}\right]^{\frac{1}{r}}$. By Theorem \ref{section:Theorem}, for any for any $0\leq\epsilon\leq2$, we have	$$
     	\begin{aligned}
     		\widetilde C_{\mathrm{NJ}}^{(p)}(l_r)&=\widetilde C_{\mathrm{NJ}}^{(p)}(L_r)\\&=\sup\left\{\frac{[2\min\{\xi,\nu\}\left[1-\left(\frac{\varepsilon}{2}\right)^{r}\right]^{\frac{1}{r}}+|\xi-\nu|]^{p}+[\min\{\xi,\nu\}\varepsilon+|\xi-\nu|]^p}{2^{p-1}(\xi^{p}+\nu^{p})}\right\}.
     	\end{aligned}
     	$$
     	
     	Now we define a function as
     	
     	$$f(\varepsilon)=[2\min\{\xi,\nu\}\left[1-\left(\frac{\varepsilon}{2}\right)^{r}\right]^{\frac{1}{r}}+|\xi-\nu|]^{p}+[\min\{\xi,\nu\}\varepsilon+|\xi-\nu|]^p.$$
     	Then by calculating the derivative of the function $f(\varepsilon)$, we obtain $f(\varepsilon)$ attains it's maximum at either $t_1=0$ or  $t'_1=2$ where $p\geq r$, and $f(\varepsilon)$ attains it's maximum at $t_2=2^\frac{1}{r}$  where $p<r$.\\Hence the conclusions of $\widetilde C_{\mathrm{NJ}}^{(p)}(l_r)$ and $\widetilde C_{\mathrm{NJ}}^{(p)}(L_r)$ hold.
     	
     \end{proof}

     \begin{Example}
     	$X$ is $c_0$ space: $c_0=\{(x_{i}):|x_{i}|\to0(i\to\infty)\}$ be furnished with the norm specified by $$\|x\|=\sup\limits_{1\leq i<\infty}\left\{|x_i|+\left(\sum\limits_{i=1}^\infty\frac{|x_i|^2}{4^i}\right)^{1/2} \right\},$$then, $$\widetilde C_\mathrm{NJ}^{p}(\xi,\nu,X)=\frac{|\xi-\nu|^{p}+(\xi+\nu)^{p}}{2^{p-1}(\xi^{p}+\nu^{p})}.$$
     \end{Example}
     \begin{proof}
     	Since $\delta_X(2)=1$ is known, we suppose $\xi\neq\nu$ without losing generality, and we obtain$$\begin{aligned}
     		\widetilde C_\mathrm{NJ}^{p}(\xi,\nu,X)&\leq\frac{[2\min\{\xi,\nu\}(1-\delta_X{(\varepsilon)}+|\xi-\nu|]^{p}+[\min\{\xi,\nu\}\varepsilon+|\xi-\nu|]^p}{2^{p-1}(\xi^{p}+\nu^{p})} \\
     		&\leq\frac{|\xi -\nu|^{p}+[2min\{\xi,\nu\}+|\xi-\nu|]^{p}}{2^{p-1}(\xi^{p}+\nu^{p})} \\
     		&=\frac{|\xi-\nu|^{p}+(\xi+\nu)^{p}}{2^{p-1}(\xi^{p}+\nu^{p})}.
     	\end{aligned}$$\\Hence by the upper bound of the $\widetilde C_\mathrm{NJ}^{p}(\xi,\nu,X)$, we obtain$$\widetilde C_\mathrm{NJ}^{p}(\xi,\nu,X)=\frac{|\xi-\nu|^{p}+(\xi+\nu)^{p}}{2^{p-1}(\xi^{p}+\nu^{p})}.$$
     	
     \end{proof}
     The following Theorem \ref{key} offers the connections between the constant $\widetilde C_\mathrm{NJ}^{p}(\xi,\nu,X)$ and $J(X)$. Nevertheless, prior to presenting the theorem, we first provide a lemma that will be required in the subsequent proof of Theorem \ref{key}.
     
     \begin{Lemma}\label{l3.1}
     	Let $x>0$, $y>0$, then $(x+y)^p\leq2^{p-1}(x^p+y^p)$ where $1\leq p<+\infty$.
     \end{Lemma}
     \begin{Theorem}\label{key}
     	Let $X$ be a Banach space. Then,$$\frac{[\max\{\xi,\nu\}]J(X)-|\xi-\nu|]^{p}}{2^{p-2}(\xi^{p}+\nu^{p})}\leq\widetilde{C}_\mathrm{NJ}^{p}(\xi,\nu,X)\leq1+\frac{[\min\{\xi,\nu\}J(X)+|\xi-\nu|]^{p}}{2^{p-1}(\xi^{p}+\nu^{p})}.$$
     \end{Theorem}
     
     \begin{proof}
     	We prove the left inequality first.
     	
     	Since 
     	$$\begin{aligned}
     		2[\min\{\|\xi x+\nu y\|,\|\nu x-\xi y\|\}]^{p}& \leq\|\xi x+\nu y\|^p+\|\nu x-\xi y\|^p 
     		\\&\leq2^{p-1}(\xi^p+\nu^p)\widetilde C_{NJ}^p(\xi,\nu,X),
     	\end{aligned}$$
     	we have$$
     	\left[2^{p-2}\left(\xi^{p}+\nu^{p}\right)\widetilde C_{NJ}^{p}(\xi,\nu,X)\right]^{\frac{1}{p}}\geq\min\{\|\xi x+\nu y\|,\|\nu x-\xi y\|.$$
     	According to
     	$$\|\xi x+\nu y\|\geq\max\{\xi,\nu\}\|x+y\|-|\xi-\nu|,$$and
     	$$\|\nu x-\xi y\|\geq\max\left\{\xi,\nu\right\}\|x-y\|-|\xi-\nu|,$$
     	we obtain$$\begin{aligned}&\left[2^{p-2}\left(\xi^{p}+\nu^{p}\right)\widetilde C_\mathrm{NJ}^{p}(\xi,\nu,X)\right]^\frac{1}{p}\\&\geq \min\{\max\{\xi,\nu\}\|x+y\|-|\xi-\nu|,\max\{\xi,\nu\}\|x-y\|-|\xi-\nu|\}\\&=\max\{\xi,\nu\}J(X)-|\xi-\nu|,\end{aligned}$$
     	which means that $\widetilde{C}_\mathrm{NJ}^{p}(\xi,\nu,X)\geq\frac{[\max\{\xi,\nu\}]J(X)-|\xi-\nu|]^{p}}{2^{p-2}(\xi^{p}+\nu^{p})}.$
     	
     	On the other hand, according to the convex function $h(x)=\|x\|^p$, we have$$\frac{\|\xi x+\nu y\|^p}{(\xi+\nu)^p}=\bigg\|\frac{\xi}{\xi+\nu}x+\frac{\eta}{\xi+\nu}y\bigg\|^p\leq\frac{\xi}{\xi+\nu}\|x\|^p+\frac{\nu}{\xi+\nu}\|y\|^p=1,$$
     	and$$\frac{\|\nu x-\xi y\|^p}{(\xi+\nu)^p}=\bigg\|\frac{\nu}{\xi+\nu}x+\frac{\xi}{\xi+\nu}(-y)\bigg\|^p\leq\frac{\xi}{\xi+\nu}\|x\|^p+\frac{\nu}{\xi+\nu}\|-y\|^p=1
     	,$$
     	which means that $\|\xi x+\nu y\|^p\leq(\xi+\nu)^p
     	$ and $\|\nu x-\xi y\|^p\leq(\xi+\nu)^p
     	$.
     	
     	By Lemma \ref{l3.1}, we obtain that \begin{equation}\label{e6}
     		\frac{\|\xi x+\nu y\|^{p}}{2^{p-1}(\xi^{p}+\nu^{p})}\leq\frac{(\xi +\nu )^{p}}{2^{p-1}(\xi^{p}+\nu^{p})}\leq\frac{2^{p-1}(\xi^{p}+\nu^{p})}{2^{p-1}(\xi^{p}+\nu^{p})}\leq1,
     	\end{equation}and\begin{equation}\label{e7}
     		\frac{\|\nu x-\xi y\|^{p}}{2^{p-1}(\xi^{p}+\nu^{p})}\leq\frac{(\xi +\nu )^{p}}{2^{p-1}(\xi^{p}+\nu^{p})}\leq\frac{2^{p-1}(\xi^{p}+\nu^{p})}{2^{p-1}(\xi^{p}+\nu^{p})}\leq1
     		.
     	\end{equation}	
     	
     	Then since \begin{equation}\label{e8}
     		\|\xi x+\nu y\|\leq\min\{\xi,\nu\}\|x+y\|+|\xi-\nu|,
     	\end{equation}and\begin{equation}\label{e9}
     		\|\nu x-\xi y\|\leq\min\{\xi,\nu\}\|x-y\|+|\xi-\nu|.
     	\end{equation} Therefore, combining the above inequality (\ref*{e6}), (\ref*{e7}), (\ref*{e8}) and (\ref*{e9}), we have$$\begin{aligned}
     		&\frac{\|\xi x+\nu y\|^{p}+\|\nu x-\xi y\|^{p}}{2^{p-1}(\xi^{p}+\nu^{p})}\\& \leq1+\frac{[\min\{\|\xi x+\nu y\|,\|\nu x-\xi y\|\}]^{p}}{2^{p-1}(\xi^{p}+\nu^{p})} \\
     		&\leq1+\frac{[\min\{\min\{\xi,\nu\}\|x+y\|+|\xi-\nu|,\min\{\xi,\nu\}\|x-y\|+|\xi-\nu|]^{p}}{2^{p-1}(\xi^{p}+\nu^{p})} \\
     		&=1+\frac{[\min\{\xi,\nu\}]J(X)+|\xi-\nu|]^{p}}{2^{p-1}(\xi^{p}+\nu^{p})},
     	\end{aligned}$$ which means that $\widetilde{C}_\mathrm{NJ}^{p}(\xi,\nu,X)\leq1+\frac{[\min\{\xi,\nu\}J(X)+|\xi-\nu|]^{p}}{2^{p-1}(\xi^{p}+\nu^{p})}.$
     	
     	In conclusion, we completed the proof of  Theorem \ref{key}.
     \end{proof}
     
     The following Theorem \ref{section:5.3} and Corollary \ref{section:6.1} provide an analogous condition that a Banach space $X$ has the property of being uniformly non-square. Particularly, the Banach space $X$ is known as uniformly non square \cite{11} if there is $\delta\in(0,1)$ such that for any $x,y\in S_X$, either $\frac{\|x+y\|}{2}\leq1-\delta$ or $\frac{\|x-y\|}{2}\leq1-\delta.$ 
     \begin{Theorem}\label{section:5.3}
     	X is uniformly non-square $\Leftrightarrow$  $\widetilde C_\mathrm{NJ}^{p}(\xi,\nu,X)<\frac{(\xi+\nu)^{p}}{2^{p-2}(\xi^{p}+\nu^{p})}.$		
     \end{Theorem}
     \begin{proof}
     	It is immediately demonstrable by the definition of uniformly non-squared presented in the article above.
     \end{proof}

     It's well known that $X$ is uniformly non-square $\Leftrightarrow$  $J(X)<2$, then we have a simple corollary by the above Theorem \ref{key} and Theorem \ref{section:5.3}.
     \begin{Corollary}\label{section:6.1}
     	Let $X$ be a Banach space. Then for any $1\leq p<+ \infty $ , the following conclusions are equivalent:
     	
     	\qquad $(i)$~~$X$ is uniformly non-square.
     	
     	\qquad $(ii)$~~$J(X)<2.$
     	
     	\qquad $(iii)$~~$\widetilde C_\mathrm{NJ}^{p}(\xi,\nu,X)<\frac{(\xi+\nu)^{p}}{2^{p-2}(\xi^{p}+\nu^{p})}.$
     	
     \end{Corollary}

     \section{ The constant $\widetilde C_\mathrm{NJ}^{p}(\xi,\nu,X)$ and uniform normal structure}
     It is known that in addition to the condition mentioned in the previous section, many geometric properties of Banach spaces are related to the fixed point property (FPP) \cite{14}, and among them, normal structure \cite{15} is one of the most widely studied properties.\\In addition to its close connection with the fixed point property, the normal structure itself has significant research value. However, it is not easy to test whether a given Banach space has a (weak) normal structure. For this reason, scholars have given sufficient conditions for this property from multiple perspectives. These include using constants such as $J(X)$ and $C_\mathrm{N J}(X)$, as well as considering aspects such as the modulus of convexity and smoothness.\\In this section, 
     presenting a sufficient condition for normal structure based on the constant $\widetilde C_\mathrm{NJ}^{p}(\xi,\nu, X)$ is a crucial step in the following Theorem \ref{6.1}. This offers a fresh perspective on the structure and characteristics of Banach spaces.
     
     \begin{Definition}
     	If K is a closed, bounded, and convex subset of a Banach space X, then X has a normal structure if and only if  $r(K)<\dim(K)$ for each K. Among them, $$\operatorname{diam}(K):=\sup\{\|x-y\|:x,y\in K\}$$ denotes the diameter,$$r(K):=\inf\{\sup\{\|x-y\|:y\in K\}:x\in K\}$$ represents the Chebyshev radius, respectively.
     \end{Definition}
     If every weakly compact convex set $K$ of Banach space $X$ that contains more than one point has a normal structure, then Banach space $X$ is said to have a weak normal structure. Normal structure and weak normal structure are obviously the same in a reflexive \cite{11} Banach space. Furthermore, the fixed point property is present in all reflexive Banach spaces with normal structure.
     The constant $\widetilde C_\mathrm{NJ}^{p}(\xi,\eta,X)$ and uniform normal structure will be connected in the following Theorem \ref{6.1}. 
     \begin{Theorem}\label{6.1}
     	Let $X$ be a Banach space, if $\widetilde C_\mathrm{NJ}^{p}(\xi,\nu,X)<\frac{(\xi+\nu)^{p}+\nu^p}{2^{p-1}(\xi^{p}+\nu^{p})}$ holds for some $\xi,\nu>0$, then$X$ has normal structure.
     \end{Theorem}
     \begin{proof}
     	First, what can be seen is that there is an equivalent definition of the constant $\widetilde C_\mathrm{NJ}^{p}(\xi,\nu, X)$ which is described as$$\widetilde C_\mathrm{N J}^p(\xi,\nu, X)=\sup \left\{\frac{\xi^p\|\ x+\frac{\nu}{\xi} y\|^p+\nu^p\|x-
     		\frac{\xi}{\nu} y\|^p}{2^{p-1}(\xi^p+\nu^p)}: x,y \in S_X\right\},$$where $1\leq p<+\infty$.
     	
     	Without loss of generality, we assume that $\xi\leq\nu$ (the case $\nu\geq\xi$ is similar). If $X$ does not have weak normal structure, by \cite{17}, for any $0<\delta<1$, there exists $s_n$ in $S_X$ with $s_n\xrightarrow{w}0$ and
     	
     	$$1-\delta<\|s_{n+1}-s\|<1+\delta $$
     	for sufficiently large $n$ and for any $s\in\operatorname{co}\{s_k\}_{k=1}^n.$
     	
     	Given that 0 is a member of the weakly closed convex hull $\{s_n\}$, which equals to the norm closed convex hull $\overline{\text{co}}\{s_n\}_{n=1}^{\infty}$, we can take $n_0\in\mathbb{N},y\in$co$\{s_n\}_{n=1}^{n_0}$ and $s^*\in\nabla_{s_1}(\nabla_{s_1})$ be supporting functional of $s_1\in S_X$, such that$$\|y\|<\delta,\quad\left|s^*\left(s_{n_0}\right)\right|<\delta,\quad1-\delta<\left\|s_{n_0}-s_1\right\|<1+\delta.$$
     	Then, for any $\frac{\xi}{\nu}\in[0,1],$
     	$$\left\|s_{n_0}-(1-\frac{\xi}{\nu})s_1\right\|\geq\left\|s_{n_0}-(1-\frac{\xi}{\nu})s_1-\frac{\xi}{\nu}y\right\|-\frac{\xi}{\nu}\|y\|>1-2\delta.$$
     	Let $x_{1}={\frac{s_{1}-s_{n_{0}}}{\|s_{1}-s_{n_{0}}\|}}, y_{1}=s_{1}$, we have$$\begin{aligned}\left\|x'-\frac{\xi}{\nu}y'\right\|& =\left\|\left(1-\frac{\xi}{\nu}\right)s_{1}-s_{n_{0}}+x'-\left(s_{1}-s_{n_{0}}\right)\right\| \\
     		&\geq\left\|s_{n_0}-\left(1-\frac{\xi}{\nu}\right)s_1\right\|-\left\|x'-\left(s_1-s_{n_0}\right)\right\| \\
     		&>1-2\delta-\begin{vmatrix}1-\begin{Vmatrix}s_1-s_{n_0}\end{Vmatrix}\end{vmatrix} \\
     		&>1-3\delta.
     	\end{aligned}$$

     	On the other hand,$$\begin{aligned}
     		\left\|x'+\frac{\nu}{\xi}y'\right\|& =\left\|\frac{s_{1}-s_{n_{0}}}{\|s_{1}-s_{n_{0}}\|}+\frac{\nu}{\xi}s_{1}\right\| \\
     		&\geq\frac{s^{*}\left(s_{1}\right)-z^{*}\left(s_{n_{0}}\right)}{\left\|s_{1}-s_{n_{0}}\right\|}+\frac{\nu}{\xi}s^{*}\left(s_{1}\right) \\
     		&>\frac{1-\nu}{1+\nu}+\frac\nu\xi\\
     		&>\left(1+\frac{\nu}{\xi}\right)-3\delta.
     	\end{aligned}$$
     	Since $\delta$ can be made extremely small at will, we obtain$$\begin{aligned}\widetilde C_\mathrm{N J}^p(\xi,\nu,X)&\geq\frac{\xi^p\|\ x'+\frac{\nu}{\xi} y'\|^p+\nu^p\|x'-
     			\frac{\xi}{\nu} y'\|^p}{2^{p-1}(\xi^p+\nu^p)}\\&\geq\frac{\xi^p(1+\frac{\nu}{\xi})^p+\nu^p}{2^{p-1}(\xi^p+\nu^p)}\\&=\frac{(\xi+\nu)^p+\nu^p}{2^{p-1}(\xi^p+\nu^p)}.\end{aligned}$$

     	It is clear that
     	
     	$$\begin{aligned}\frac{(\xi+\nu)^p+\nu^p}{2^{p-1}(\xi^p+\nu^p)}<\frac{(\xi+\nu)^{p}}{2^{p-2}(\xi^{p}+\nu^{p})}\end{aligned}$$
     	for $\xi,\nu>0$, as well as Corollary \ref{section:6.1}, which demonstrates that $X$ is uniformly non-square and therefore reflexive. We can conclude that $X$ has a normal structure since as we already discussed normal structure and weak normal structure coincide. 
     	
     \end{proof}

\end{document}